\documentclass[amsmath,amssymb,amsthm,aps]{amsart}
\usepackage[english]{babel}
\usepackage[utf8]{inputenc}

\usepackage{amsfonts}
\usepackage{graphicx}
\usepackage{gensymb}
\usepackage[colorinlistoftodos]{todonotes}
\usepackage{caption}
\usepackage{subcaption}
\usepackage{enumerate}

\usepackage{hyperref}

\newtheorem{corollary}{Corollary}[section]
\newtheorem{theorem}{Theorem}[section]
\newtheorem{lemma}{Lemma}[section]
\newtheorem{definition}{Definition}[section]

\numberwithin{equation}{section}

\theoremstyle{remark}

\begin{document}
\pagestyle{empty}
\title{Honors Thesis: On the faithfulness of the Burau representation at roots of unity}


\author{Thomas Chuna}
\address{Michigan State University}
\curraddr{}
\email{chunatho@msu.edu}
\thanks{A great deal of thanks to my advisors on this project, A. Parker from Wittenberg University and B. Winarski from University of Wisconsin-Milwaukee, for their help and guidance. You were both instrumental in the completion of my degree at Wittenberg University.}



\subjclass[2010]{Primary}

\date{5/26/16}


\begin{abstract}
We study the kernel of the evaluated Burau representation through the braid element $\sigma_1 \sigma_2 \sigma_1$. This element is significant as a part of the standard braid relation. We establish the form of this element's image raised to the $k^{th}$ power. Interestingly, the cyclotomic polynomials arise and can be used to define the expression. The main result of this paper is that the Burau representation of the braid group of $n$ strands for $n > 2$ is unfaithful at any $\tau^{th}$ primitive root of unity, such that $\tau$ is greater than $3$.
\end{abstract}

\maketitle

\section{Introduction}
\subsection{Mapping Class Groups}
By the hyper-elliptic involution, the disk of $3$ marked points, $D_3$, is the quotient of the closed orientable surface of genus one with one boundary, $\chi_1^1$. Thus $\chi_1^1$ serves as a two fold branched covering space of $D_3$. Not only are the surfaces related by a two fold branched cover, but their mapping class groups are related. Under the Birman-Hilden theorem the mapping class group of the disk is isomorphic to a subgroup of the Dehn Twists, (\textit{i.e.} $\chi_1^1$'s mapping class group). The mapping class group of $D_n$ is isomorphic to the braid group of $n$ strands \cite{3}. Artin's braid relations, defined on the braid group, are similar to relations defined on the Dehn twists. The map $\psi$ transfers the braid relation and the disjointness relation from $\text{MOD}(\chi_g^1)$ to the braid group of $2g+1$ marked points \cite{2}. Pictorially, these maps are presented in figure \ref{fig:Overview}.

\begin{figure}[h]
\centering
\includegraphics[width=0.5\textwidth]{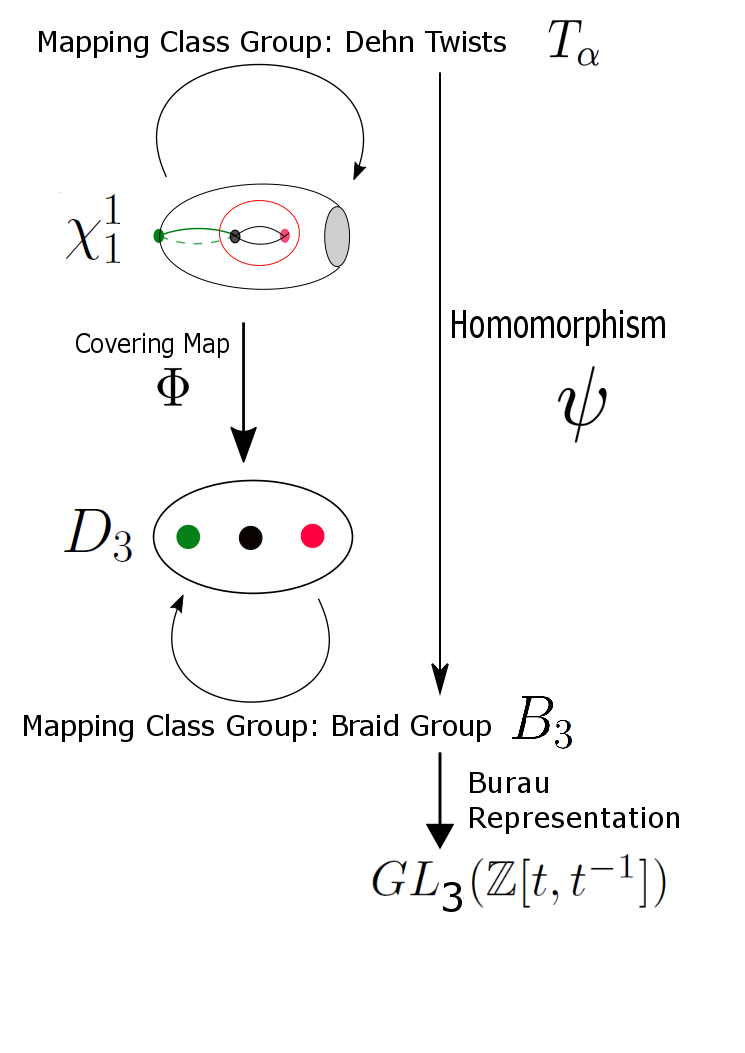}
\caption{\label{fig:Overview} A representation of the maps between topological surfaces.}
\end{figure}

\subsection{The Braid Group}
Let $B_n$ denote the braid group of $n$ strands. Let $\sigma_1 ,..., \sigma_{n-1}$ denote the standard generators of $B_n$ and $\sigma_1^{-1} ,\dots, \sigma_{n-1}^{-1}$ denote their inverses \ref{fig:genOfBn}. The Artin representation of the braid group of $n$-strands is defined \cite{1}:
\begin{definition}
$B_n =  \{ \sigma_1 ,..., \sigma_{n-1} \mid \sigma_i \sigma_j =  \sigma_j \sigma_i \, \text{ if } \, \lvert i - j \rvert> 1 \, , \, \sigma_i \sigma_{i+1} \sigma_i = \sigma_{i+1} \sigma_{i} \sigma_{i+1} \, \text{ for } \, 1 \le i \le n-2 \}$
\end{definition}
If left and right directions of \ref{fig:MOD(D_n)} are conceptualized as backward and forward motion in time, respectively, then the strands are tracing the motion of the marked points in the disk through time. The homeomorphism of $D_3$ presented in figure \ref{fig:MOD(D_n)} can be expressed as the combination of four generators, $\sigma_2^{-1} \sigma_2^{-1} \sigma_1 \sigma_2^{-1}$. Note these braid generators are written using standard functional composition notation (\textit{i.e.} $\sigma_1$ is applied second).
\begin{figure}[h]
\begin{subfigure}{.5\textwidth}
\centering
\includegraphics[width=0.7\textwidth]{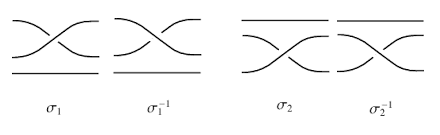}
\caption{\label{fig:genOfBn}The standard generators of $B_n$}
\label{fig:sub2}
\end{subfigure}
\centering
\begin{subfigure}{.5\textwidth}
\centering
\includegraphics[width=0.7\textwidth]{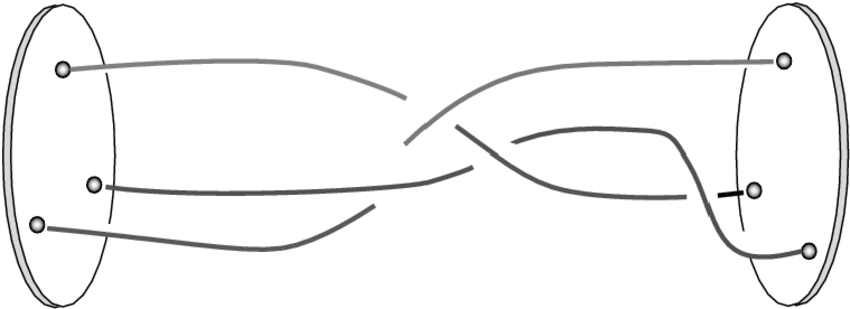}
\caption{\label{fig:MOD(D_n)} A homeomorphism, $\sigma_2^{-1} \sigma_1 \sigma_2^{-1} \sigma_2^{-1}$, from MOD($D_{n}$)}
\end{subfigure}
\caption{Expressing elements of SMOD($D_3$) as elements of $B_3$. [5][6]}
\label{fig:test}
\end{figure}

\section{Algebraic Structure under the Burau Representation}
As seen the braid group is a quantity of interest in group geometric theory. The Burau representation maps the braid group of $n$ strands to $GL_{n}\left(\mathbb{Z}[t,t^{-1}]\right)$. This maps is shown at the bottom of figure \ref{fig:Overview}. In this work, we study properties of mapping class groups via the general linear group. Let us define the Burau representation.
\begin{definition}
Let $I_k$ denote the $k \, \times \, k$ identity matrix. The Burau representation of the braid group is the map:
\begin{align*}
\rho : B_n \rightarrow GL_{n}(\mathbb{Z}[t,t^{-1}]) && \sigma_i \rightarrow I_{i-1} \bigoplus \begin{bmatrix} 
    1-t&t\\1&0 \end{bmatrix} \bigoplus I_{n-i-1}
\end{align*}
\end{definition} 
The algebraic structure created by the image of $\sigma_1 \sigma_2 \sigma_1$ under the Burau representation, seen in equation \ref{equ:zeta}, provides insight into the kernel of the evaluated Burau representation. The element $\sigma_1 \sigma_2 \sigma_1$ is selected for study because of its novelty in Artin's presentation of the braid group. However, other than exposing the element, the braid relation itself is not used in this work.
\begin{equation} \label{equ:zeta}
\rho (\sigma_{1} \sigma_2 \sigma_1) = \zeta = \begin{bmatrix} 1-t&(1-t)t&t^{2}\\1-t&t&0\\1&0&0& \end{bmatrix} 
\end{equation}
Immediately, there is nothing notable about this image. Regardless, as we shall show, the cyclotomic polynomials are exposed in $\zeta^k$ for even $k$. $\zeta^k$, has two different forms, one for odd $k$ and one for even $k$. The algebraic structure that arises in the even powers is enough to prove that the Burau representation is unfaithful when evaluated at any $\tau^{th}$ primitive root of unity, such that $\tau >3$. Using proof by induction, we establish the form for even $k$.
\begin{lemma} 
If $k \in \mathbb{N}$ such that $k$ is even, then the image of $(\sigma_{1} \sigma_2 \sigma_1)^k$ in $B_3$, under the Burau representation is:
\begin{equation}\label{equ:zetak}
\zeta^k = \begin{bmatrix} 
a_{\frac{3}{2}k}&t a_{\frac{3}{2}k-2}&t^2 a_{\frac{3}{2}k-2} 
\\a_{\frac{3}{2}k-2}&t a_{\frac{3}{2}k-1}&t^2 a_{\frac{3}{2}k-2}
\\a_{\frac{3}{2}k-2}&t a_{\frac{3}{2}k-2}&t^2 a_{\frac{3}{2}k-3}
\end{bmatrix}
\end{equation}
Let $\phi_i$ be the $i^{th}$ cyclotomic polynomial and $Z = \{ \, d \in \mathbb{N} \, \vert \, d \text{ divides } m+2 \, \}$. Define the polynomial $a_m$ such that:
\begin{align}\label{equ:am}
a_m &= \begin{cases} 
        \frac{1+t^{m+1}+t^{m+2}}{\phi_3},  \forall m \equiv 0 \text{ mod } 3
      \\  - \frac{1}{\phi_3} \prod\limits_{d \in Z}\phi_d  = \frac{1-t^{m+2}}{\phi_3}, \forall m \equiv 1 \text{ mod } 3
      \\  \frac{1+t^{m}+t^{m+2}}{\phi_3}, \forall m \equiv  2 \text{ mod } 3
\end{cases}
\end{align}
\end{lemma}
\begin{proof}
Assume $k \in \mathbb{N}$ such that $k$ is even and consider the base case where $k=2$, then:
\begin{equation*}
\zeta^2 = \begin{bmatrix} 
  \frac{1}{\phi_3}( 1+t^4+t^5)&\frac{1}{\phi_3}(t(1-t^3))&\frac{1}{\phi_3}(t^2(1-t^3)) \\\frac{1}{\phi_3}(1-t^3)&\frac{1}{\phi_3}(t+t^3+t^5)&\frac{1}{\phi_3}(t^2(1-t^3))
\\\frac{1}{\phi_3}(1-t^3)&\frac{1}{\phi_3}(t(1-t^3))&\frac{1}{\phi_3}(t^2(1+t+t^2)  \end{bmatrix} = \begin{bmatrix}a_3&t a_1&t^2 a_1\\a_1&t a_2&t^2 a_1\\a_1&t a_1&t^2 a_0\end{bmatrix}
\end{equation*}
Assuming the induction hypothesis, we examine the form of $\zeta^{k+2}$.
\begin{equation*}
\zeta^{k+2} = \zeta^k*\zeta^2 = \begin{bmatrix} 
a_{\frac{3}{2}k}&t a_{\frac{3}{2}k-2}&t^2 a_{\frac{3}{2}k-2} 
\\a_{\frac{3}{2}k-2}&t a_{\frac{3}{2}k-1}&t^2 a_{\frac{3}{2}k-2}
\\a_{\frac{3}{2}k-2}&t a_{\frac{3}{2}k-2}&t^2 a_{\frac{3}{2}k-3}
\end{bmatrix}*\begin{bmatrix}a_3&t a_1&t^2 a_1\\a_1&t a_2&t^2 a_1\\a_1&t a_1&t^2 a_0\end{bmatrix}
\end{equation*}
Consider element [1,1] of the matrix $\zeta^k$, since $k$ is even then $\frac{3}{2}k \equiv 0$ mod $3$, $\frac{3}{2}k-1 \equiv 2$ mod $3$ and $\frac{3}{2}k-3 \equiv 0$ mod $3$. Explicitly multiplying the first row and the first column we get the $[1,1]$ element of the matrix.
\begin{align*}
[1,1] &= a_{\frac{3}{2}k}*a_3 + t a_{\frac{3}{2}k-2}*a_1 + t^2 a_{\frac{3}{2}k-2}*a_1
\\
&= \frac{1}{\phi_3}(1-t +t + t^2 - t^2 +t^3  - t^3 + t^{\frac{3}{2}k+1} -t^{\frac{3}{2}k+1} +t^{\frac{3}{2}k+2}-t^{\frac{3}{2}k+2}
\\& + t^{\frac{3}{2}k+2}- t^{\frac{3}{2}k+2} + t^{\frac{3}{2}k+3} -t^{\frac{3}{2}k+3}  +t^{\frac{3}{2}k+4}+t^{\frac{3}{2}k+5})
\\&= \frac{1}{\phi_3}(1+t^{\frac{3}{2}k+4}+t^{\frac{3}{2}k+5}),  \intertext{\textit{Using the definition of $a_m$}, from equation \ref{equ:am}.}
&= a_{\frac{3}{2}(k+2)}
\end{align*}
We see that the conjectured form of $\zeta^{k+2}$'s [1,1] matrix element holds. Similar arguments can be used to prove each of the other 8 elements of the even form. 
\end{proof}
\section{The Burau Representation's Kernel}
Now that the general form of $\zeta^{k}$ is established or even $k$. We set a constraint on the system, $\zeta^{k}=I$, to explore what roots yield the identity matrix. Given this constraint, it will be shown that any $\tau^{th}$ primitive root of unity, such that $\tau > 3$ is the solution to $\zeta^{k}=I$ for some even $k$. Furthermore, in corollary \ref{crl:root}, we determine the minimum $k$ such that the $\tau^{th}$ root of unity is a solution to $\zeta^{k}=I$.
\begin{theorem} \label{thrm:kernel}
If $t = e^{\frac{2 \pi \imath}{\tau}}$ then for any given $\tau \in \mathbb{N}$, such that $\tau > 3$, there exists $k \in \mathbb{Z}$ such that $\zeta^k = I_3$.
\end{theorem}
\begin{proof} 
Let $k$ be even, first consider the off-diagonal terms. since then $\frac{3}{2}k-2 \equiv 0$ mod $3$, the form of equation \ref{equ:zetak}, implies that the off-diagonal elements are proportional to $a_{\frac{3}{2}k-2}$. Thus, for $\zeta^k = I_3$ the definition \ref{equ:am} implies the off-diagonal elements yield $\frac{1-t^{\frac{3k}{2}}}{\phi_3} =0$. After redistributing terms we see that as long as $t$ is not evaluated at the roots of $\phi_3$ then:
\begin{equation} \label{equ:offdiagonal}
t^{\frac{3k}{2}} = 1
\end{equation}

We shall use equation \ref{equ:offdiagonal} to show that $\zeta^k = I$.  Consider the diagonal elements of the even form. Since $k$ is even then $\frac{3}{2}k \equiv 0$ mod $3$, $\frac{3}{2}k-1 \equiv 2$ mod $3$ and $\frac{3}{2}k-3 \equiv 0$ mod $3$. Thus the diagonal elements can be evaluated as: 

\begin{align*}
a_{\frac{3}{2}k} &=  \frac{1+t^{\frac{3}{2}k+1}+t^{\frac{3}{2}k+2}}{\phi_3} =  \frac{ 1+(t)+(t^2)}{\phi_3}  = \frac{1+t+t^2}{\phi_3} 
\\&= 1
\\t a_{\frac{3}{2}k-1} &= t * \frac{1+t^{\frac{3}{2}k-1}+t^{\frac{3}{2}k+1}}{\phi_3} = t* \frac{1+(t^{-1})+(t)}{\phi_3}=\frac{1+ t +t^2}{\phi_3}
\\&= 1
\\t^2 a_{\frac{3}{2}k-3} &= t^2 * \frac{1+t^{\frac{3}{2}k-2}+t^{\frac{3}{2}k-1}}{\phi_3} = t^2 * \frac{1+t^{-2} + t^{-1}}{\phi_3} = \frac{1+t+t^2}{\phi_3}
\\&= 1
\end{align*}
In conclusion, if $t^{\frac{3k}{2}} = 1$ then $\zeta^k = I_3$. 

Additionally, it was assumed that $t = e^{\frac{2 \pi \imath}{\tau}}$ this implies $t^{\tau}=1$. Comparing this to the off-diagonal constraint, $t^{\frac{3}{2}k}=1$. Both equations are satisfied if and only if $\frac{3}{2}k = j*\tau$ where $j \in \mathbb{Z}$. Thus, $k$ must be divisible by $2$ and $\tau$. Trivially, this is satisfied by $k = 2 \tau$. Therefore, given any integer $\tau$ in $t = e^{\frac{2 \pi \imath}{\tau}}$ there exists a $k$ such that $\zeta^k = I_3$.
\end{proof}
\begin{corollary} \label{crl:root}
For a given root of unity, $\tau$, the minimum power, $k$, that makes $(\sigma_1 \sigma_2 \sigma_1)^k$ equal to the identity matrix is  $k=\frac{2}{3}\tau$ for either $\tau \equiv 0 \, \text{mod} \, 3$ and $k=2\tau$ for $\tau \equiv 1 \, \text{mod} \, 3$ or $\tau \equiv 2 \, \text{mod} \, 3$.
\end{corollary}
\begin{proof}
In theorem \ref{thrm:kernel} we encountered two conditions. First, that $t^{\frac{3k}{2}} = 1$ and second that $t = e^{\frac{2 \pi \imath}{\tau}}$ which implied that $t^{\tau}=1$. Together these constraints showed that $t^{\frac{3}{2}k}=1$ when $\frac{3}{2}k = j*\tau$ such that $j \in \mathbb{N}$ (i.e. $\frac{3}{2}k \text{ is a multiple of } \tau$). We will now classify the minimal $k$ via $\tau \text{ mod } 3$:
\begin{enumerate}
\item $\tau = 3l$

Given $\tau$ is also equal to $\frac{3}{2}k$, then $3l = \frac{3}{2}k$. Since both sides are divisible by $3$, $l = \frac{k}{2}$. Therefore $k$ must be divisible by $2$ and $l$. The smallest term divisible by both $2$ and $l$ is $2l$. Given that $3l=\tau$, then $k=\frac{2}{3}\tau$ is the minimum $k$ which yields the identity.

\item $\tau = 3l+1$ or $\tau = 3l+2$ 

Given $\tau$ is also equal to $\frac{3}{2}k$, then $3l+1 = \frac{3}{2}k$ or $3l+2 = \frac{3}{2}k$. Neither value of $\tau$ is divisible by $3$. Therefore k must be divisible by both $2$ and $\tau$. The smallest term divisible by both $2$ and $\tau$ is $2 \tau$, thus $k = 2\tau$ is the minimum $k$ which yields the identity.
\end{enumerate} 
\end{proof}
It has now been shown that two distinct elements map to the identity matrix when the Burau representation is evaluated at any primitive root of unity excepting the first three. This is a novel result and provides the necessary strength to show the unfaithfulness of Burau representation. In the above proofs, we used the specific element $\sigma_1 \sigma_2 \sigma_1 \in B_3$. Therefore, once we have shown the unfaithfulness for $B_3$ in corollary \ref{crl:B3faith}, the result will be generalized to $B_n \text{,  } \forall n \geq 3$ in corollary \ref{crl:Bnfaith}.

\begin{corollary}\label{crl:B3faith}
If the Burau representation of $B_3$ is evaluated at any primitive root of unity, excepting the first three, then the representation is unfaithful.
\end{corollary}
\begin{proof}
In theorem \ref{thrm:kernel} we showed that if $t$ is evaluated at any root of unity,  excepting the first three, then there exists a $k$ such that $\zeta^k = I_3$. This implies the kernel of the Burau representation is non-trivial. Therefore the map is unfaithful.
\end{proof}
\begin{corollary}\label{crl:Bnfaith}
For all $n \ge 3$, the Burau representation of $B_n$ evaluated all roots of unity  excepting the first three is unfaithful. 
\end{corollary}
\begin{proof}
We must consider that $\sigma_1 \sigma_2 \sigma_1$ does not exist in $B_1$ or $B_2$. Limiting ourselves to $n \geq 3$, the image of $\sigma_i \sigma_{i+1} \sigma_i \in B_n$ is expressed as the segmented matrix:
\begin{equation}
\zeta = \begin{bmatrix} \begin{array}{c|c|c}
I_{i-1} & 0 & 0 \\   \hline
0& \begin{smallmatrix} 1-t&(1-t)t&t^{2} \\1-t&t&0\\1&0&0 \end{smallmatrix} & 0 \\ 
  \hline
  0 & 0 & I_{n-(i+2)}
 \end{array} \end{bmatrix}
\end{equation}
When this image is expressed as a segmented matrix it is easy to see that under matrix multiplication $\zeta^k =  I_{i-1} \bigoplus \begin{bmatrix} 1-t&(1-t)t&t^{2} \\1-t&t&0\\1&0&0 \end{bmatrix}^k  \bigoplus I_{n-(i+2)}$.
This expression clearly holds the same general form from equation \ref{equ:zetak}, which was proven in theorem \ref{thrm:kernel}. Thus, corollary \ref{crl:B3faith} extends to all braid groups with three or more strands.
\end{proof}

\section{Results}
In this paper, we studied the braid group using representation theory. It was established in theorem \ref{thrm:kernel} that the image of $\sigma_1 \sigma_2 \sigma_2$ generates a cyclic group. The periodicity of the generated group was explicitly determined in corollary \ref{crl:root}. Using this periodicity it was shown that the evaluated Burau representation of any braid group which contains the braid relation is unfaithful when evaluated at primitive roots of unity excepting the first three (corollary \ref{crl:Bnfaith}). Lastly, this unfaithfulness was extended to all braid groups which contain $\sigma_1 \sigma_2 \sigma_2$.

\bibliographystyle{amsplain}

\end{document}